\documentclass[reqno]{amsart}
\usepackage{amsmath}
\usepackage{amssymb}
\usepackage{amscd}
\usepackage{color}
\newtheorem{theorem}{Theorem}[section]
\newtheorem{lemma}[theorem]{Lemma}
\newtheorem{corollary}[theorem]{Corollary}
\newtheorem{proposition}[theorem]{Proposition}

\newtheorem{remark}[theorem]{Remark}

\begin{document}
\title[New Inequalities for $q$-ary Constant-Weight Codes]{New Inequalities for $q$-ary Constant-Weight Codes}

\author[Hyun Kwang Kim and Phan Thanh Toan]{Hyun Kwang Kim and Phan Thanh Toan}

\address[Hyun Kwang Kim]{Department of Mathematics\\
Pohang University of Science and Technology\\ Pohang 790-784,
Republic of Korea}
\email{hkkim@postech.ac.kr}

\address[Phan Thanh Toan]{Department of Mathematics\\
Pohang University of Science and Technology\\ Pohang 790-784,
Republic of Korea}
\email{pttoan@postech.ac.kr}



\begin{abstract}
Using double counting, we prove Delsarte inequalities for $q$-ary codes and their improvements. Applying the same technique to $q$-ary constant-weight codes, we obtain new inequalities for $q$-ary constant-weight codes.
\end{abstract}

\maketitle

\section{Introduction}

Let $\mathbb{F}_q$ be a finite field with $q$ elements and let $n$ be a positive integer. The {\em (Hamming) distance} between two vectors $u$ and $v$ in $\mathbb{F}_q^n$, denoted by $d(u,v)$, is the number of coordinates where they differ. A {\em $q$-ary code} of {\em length} $n$ is a subset of $\mathbb{F}_q^n$. If $\mathcal C$ is a $q$-ary code (of length $n$), then each element of $\mathcal C$ is called a {\em codeword}. The {\em size} of $\mathcal C$, denoted by $|\mathcal C|$, is the number of codewords in $\mathcal C$. Let $\mathcal C$ be a $q$-ary code of length $n$. The {\em distance distribution} $\{B_i\}_{i=0}^n$ of $\mathcal C$ is defined by
\begin{eqnarray}
B_i = \frac{1}{|\mathcal C|} \cdot |\{(u,v) \mid u, v \in \mathcal C, d(u,v) = i\}|.
\end{eqnarray}
It is well known that for each $k = 1, 2, \ldots, n$,
\begin{eqnarray}\label{del}
\sum_{i=0}^n P_k(n;i) B_i \geq 0,
\end{eqnarray}
where $P_k(n;x)$ is the {\em Krawtchouk polynomial} given by
\begin{eqnarray}
P_k(n;x) = \sum_{j=0}^k (-1)^j (q-1)^{k-j} \left(x \atop j \right) \left(n-x \atop k-j \right).
\end{eqnarray}
The inequalities (\ref{del}) are called the {\em Delsarte inequalities}, which were proved by Delsarte in 1973 \cite{d}. These equalities are extremely useful in coding theory (they are used to give upper bounds on sizes of $q$-ary codes through linear programming \cite{d,ms}).

When $q=2$, using double counting, Kang, Kim, and Toan were able to prove simultaneously the Delsarte inequalities and two known improvements when the size of the binary code is odd and congruent to $2$ modulo $4$, respectively \cite{kkt}. Applying the same technique to binary constant-weight codes, they obtained new linear inequalities, which allowing them to give new upper bounds on sizes of binary constant-weight codes.

The purpose of this paper is to generalize the results in \cite{kkt} to arbitrary $q$-ary codes. In Section \ref{ss2}, we prove simultaneously Delsarte inequalities for $q$-ary codes and their improvements (Theorem \ref{delthm}). Applying the same technique to $q$-ary constant-weight codes, in Section \ref{ss3}, we obtain new inequalities for $q$-ary constant-weight codes (Theorem \ref{maintr}). These inequalities generalize inequalities shown by \"{O}sterg\aa rd in \cite{o}.

\section{Delsarte Inequalities for $q$-Ary Codes}\label{ss2}

In this section, we prove the Delsarte inequalities for $q$-ary codes and their improvements. For the improvements of the Delsarte inequalities, see \cite{as}. A simple proof of the Delsarte inequalities appeared in \cite{sv}. For two vectors $a = (a_1, a_2, \ldots, a_j)$ and $b = (b_1, b_2, \ldots, b_j)$ in $\mathbb{F}_q^j$, the inner product of $a$ and $b$ is defined by $a \cdot b = a_1 b_1 + a_2b_2 + \cdots + a_j b_j.$
Let $\mathbb{F}_q^*$ be the set of nonzero elements in $\mathbb{F}_q$. For each $a = (a_1, a_2, \ldots, a_j) \in (\mathbb{F}_q^*)^j$, denote
\begin{eqnarray}
N(a) = |\{b = (b_1, b_2, \ldots, b_j) \in (\mathbb{F}_q^*)^j \mid a \cdot b \not = 0\}|
\end{eqnarray}
and
\begin{eqnarray}
Z(a) = |\{b = (b_1, b_2, \ldots, b_j) \in (\mathbb{F}_q^*)^j \mid a \cdot b  = 0\}|.
\end{eqnarray}

\begin{proposition}\label{prop21}
For each $a = (a_1, a_2, \ldots, a_j) \in (\mathbb{F}_q^*)^j$,
\begin{eqnarray}\label{mot1}
N(a) = \frac{q-1}{q} [ (q-1)^j - (-1)^j]
\end{eqnarray}
and
\begin{eqnarray}\label{hai2}
Z(a) = (q-1)^j - \frac{q-1}{q} [ (q-1)^j - (-1)^j].
\end{eqnarray}
\end{proposition}

\begin{proof}
The proof is straightforward by using induction on $j$.
\end{proof}

\begin{lemma}\label{lm1cot}
Suppose that $n_1 + n_2 + \cdots + n_h = M$, where $M$ is a constant and $n_c$ $(c=1,2,\ldots,h)$ are nonnegative integers. Then the sum
\begin{eqnarray}
\sum_{c<d} n_c n_d
\end{eqnarray}
is maximum if and only if $|n_c - n_d| \leq 1$ for all $c \not = d$.
\end{lemma}

\begin{proof}
Suppose that there exist $c_0 \not = d_0$ (not necessary $c_0 < d_0$) such that $n_{c_0} - n_{d_0} >1$. Let $n'_e = n_e$ if $e \not \in \{c_0,d_0\}$, $n'_{c_0} = n_{c_0} -1$, and $n'_{d_0} = n_{d_0} +1$. Then $n'_1 + n'_2 + \cdots + n'_h = M$ and
\begin{eqnarray}
\nonumber \sum_{c<d}n'_c n'_d - \sum_{c<d}n_cn_d & =& (n_{c_0} -1)(n_{d_0}+1) - n_{c_0}n_{d_0}\\
& = &  n_{c_0} - n_{d_0} - 1 > 0.
\end{eqnarray}
\end{proof}

Let $\mathcal C$ be a $q$-ary code of length $n$ with distance distribution $\{B_i\}_{i=0}^n$ and let $M = |\mathcal C|$. Consider $\mathcal C$ as an $M \times n$ matrix (where each $c \in \mathcal C$ is a row). The $m$th row of $\mathcal C$ is denoted by $(c_{m1}, c_{m2}, \ldots, c_{mn})$, $m = 1, 2, \ldots, M$. Let $u'_1, u'_2, \ldots, u'_n$ be the $n$ columns of $\mathcal C$. Each $u'_i$ can be considered as a vector in $\mathbb{F}_q^M$. Write $\mathbb{F}_q = \{0=\omega_1, \omega_2, \ldots, \omega_q\}$. For each $c =1, 2, \ldots, q$ and each vector $a = (a_1, a_2, \ldots, a_M) \in \mathbb{F}_q^M$, denote
\begin{eqnarray}
x_c(a) = |\{ j \mid a_j = \omega_c\}|.
\end{eqnarray}
By definition, $\sum_{c=1}^q x_c(a) = M$. For each $k = 1, 2, \ldots, n$, let
\begin{eqnarray}\label{tm10}
S(k)=\sum_{\substack{\alpha \in (\mathbb{F}_q^*)^k \\ i_1< \cdots < i_k}} \sum_{\substack{c<d}} x_c(\alpha_1 u'_{i_1} + \cdots + \alpha_k u'_{i_k}) \cdot x_d(\alpha_1 u'_{i_1} + \cdots + \alpha_k u'_{i_k}).
\end{eqnarray}

\begin{lemma}\label{qt1}
Suppose that $\mathcal C$ is a $q$-ary code of length $n$ with size $M$. Then for each $k = 1, 2, \ldots, n$,
\begin{eqnarray}
S(k) \leq (q-1)^k \left(n \atop k \right) \left[ \frac{q-1}{2q} M^2 + \frac{r(r-q)}{2q} \right],
\end{eqnarray}
where $r$ is the remainder when dividing $M$ by $q$.
\end{lemma}

\begin{proof}
Write $M = sq + r$, where $s$ is an integer and $0 \leq r < q$. For each $\alpha \in (\mathbb{F}_q^*)^k$ and each relation $i_1< i_2 < \cdots < i_k$, we always have
\begin{eqnarray}
\sum_{c=1}^q x_c(\alpha_1 u'_{i_1} + \cdots + \alpha_k u'_{i_k}) = M.
\end{eqnarray}
By Lemma \ref{lm1cot}, the sum
\begin{eqnarray}
\sum_{c<d} x_c(\alpha_1 u'_{i_1} + \cdots + \alpha_k u'_{i_k}) \cdot x_d(\alpha_1 u'_{i_1} + \cdots + \alpha_k u'_{i_k})
\end{eqnarray}
is maximum when
\begin{eqnarray}
x_e(\alpha_1 u'_{i_1} + \cdots + \alpha_k u'_{i_k}) = s
\end{eqnarray}
for $q-r$ values of $e$ and
\begin{eqnarray}
x_e(\alpha_1 u'_{i_1} + \cdots + \alpha_k u'_{i_k}) = s+1
\end{eqnarray}
for the other $r$ values of $e$. This means
\begin{eqnarray}
\nonumber && \sum_{c<d} x_c(\alpha_1 u'_{i_1} + \cdots + \alpha_k u'_{i_k})\cdot x_d(\alpha_1 u'_{i_1} + \cdots + \alpha_k u'_{i_k})\\
\nonumber & \leq & \left(q-r \atop 2\right) s^2 + (q-r)r s(s+1) + \left(r \atop 2 \right) (s+1)^2\\
\nonumber & = & \left(q-r \atop 2\right) \left(\frac{M-r}{q}\right)^2 + (q-r)r \frac{M-r}{q}\left(\frac{M-r}{q}+1\right) + \left(r \atop 2 \right) \left(\frac{M-r}{q}+1\right)^2\\
\nonumber & = & \frac{q-1}{2q} M^2 + \frac{r(r-q)}{2q}.
\end{eqnarray}
The result then follows since $|(\mathbb{F}_q^*)^k| = (q-1)^k$ and there are exactly $\left( n \atop k\right)$ choices for $i_1 < i_2 < \cdots < i_k$.
\end{proof}

For each $k = 1, 2, \ldots, n$, we introduce the polynomials
\begin{eqnarray}
P_k^-(n;x) = \frac{1}{2} \sum_{j=0}^k [(q-1)^j - (-1)^j] (q-1)^{k-j}\left(x \atop j \right) \left(n-x \atop k-j\right)
\end{eqnarray}
and
\begin{eqnarray}
P_k^+(n;x) = (q-1)^k \left(n \atop k\right) - P_k^-(n;x).
\end{eqnarray}

\begin{remark}
{\rm It follows from definition that $P_k^+(n;x) + P_k^-(n;x) = (q-1)^k \left(n \atop k\right)$ and $P_k^+(n;x) - P_k^-(n;x) = \sum_{j=0}^k (-1)^j (q-1)^{k-j}\left(x \atop j \right) \left(n-x \atop k-j\right) = P_k(n;x)$ is the Krawtchouk polynomial.}
\end{remark}

\begin{lemma}\label{qt2}
Suppose that $\mathcal C$ is a $q$-ary code of length $n$ with size $M$ and distance distribution $\{B_i\}_{i=0}^n$. Then for each $k = 1, 2, \ldots, n$,
\begin{eqnarray}
\sum_{i=1}^n P_k^-(n;i) B_i = \frac{q}{(q-1)M}S(k)
\end{eqnarray}
and
\begin{eqnarray}
\sum_{i=1}^n P_k^+(n;i) B_i = (M-1)(q-1)^k \left(n \atop k \right) - \frac{q}{(q-1)M}S(k).
\end{eqnarray}
\end{lemma}

\begin{proof}
Write $\mathcal C = (c_{mi})$. Let $S_1(k)$ be the number of pairs $(A,\alpha)$ satisfying the following conditions.
\begin{itemize}
\item [(i)] $A$ is a $2 \times k$ matrix
\begin{eqnarray}
A = \left(
\begin{array}{cccc}
c_{mi_1} & c_{mi_2} & \cdots & c_{mi_k} \\
c_{li_1} & c_{li_2} & \cdots & c_{li_k}
\end{array}
\right)
\end{eqnarray}
such that $m \not = l $ and $i_1 < i_2 < \cdots < i_k$.
\item [(ii)] $\alpha = (\alpha_1, \alpha_2, \ldots, \alpha_k)$ is an element in $(\mathbb{F}_q^*)^k$ such that
\begin{eqnarray}
\alpha_1 (c_{mi_1}-c_{li_1}) + \cdots + \alpha_k (c_{mi_k}-c_{li_k}) \not = 0.
\end{eqnarray}
\end{itemize}
For two row $u = (c_{m1}~c_{m2} \cdots c_{mn})$ and $v = (c_{l1}~c_{l2} \cdots c_{ln})$, we first choose a set $I_1$ containing $j$ coordinates ($0 \leq j \leq d(u,v)$) where $u$ and $v$ differ and choose another set $I_2$ containing $k-j$ coordinates where $u$ and $v$ are the same. Let $\{i_1, i_2, \ldots, i_k\}$ be such that $\{i_1, i_2, \ldots, i_k\}=I_1 \cup I_2$ and $i_1 < i_2 < \cdots < i_k$. For each $j$, there are exactly
$\left(d(u,v) \atop j \right) \left(n-d(u,v) \atop k-j \right)$
choices for such $i_1 < i_2 < \cdots < i_k$. Now fix $u, v$, $j$, and $i_1 < i_2 < \cdots < i_k$. We count the number of $\alpha = (\alpha_1, \alpha_2, \ldots, \alpha_k) \in (\mathbb{F}_q^*)^k$ such that
$\alpha_1 (c_{mi_1}-c_{li_1}) + \cdots + \alpha_k (c_{mi_k}-c_{li_k}) \not = 0$
or
\begin{eqnarray}\label{hmuoi20}
\sum_{t \in I_1} \alpha_t (c_{mi_t} - c_{li_t}) \not = 0.
\end{eqnarray}
By Proposition \ref{prop21}, there are exactly $\frac{q-1}{q} [(q-1)^j - (-1)^j]$ choices for $(\alpha_t)_{i_t \in I_1}$ such that (\ref{hmuoi20}) holds. Since $(\alpha_t)_{i_t \in I_2}$ can be chosen arbitrarily, we get in total
\begin{eqnarray}
\frac{q-1}{q} [(q-1)^j - (-1)^j] \cdot (q-1)^{k-j}
\end{eqnarray}
choices for $\alpha$. In conclusion,
\begin{eqnarray}\label{mm}
\nonumber S_1(k) & = & \sum_{\substack{u, v \in \mathcal C\\u \not = v}} \sum_{j=0}^{d(u,v)} \frac{q-1}{q} [(q-1)^j - (-1)^j](q-1)^{k-j}\left(d(u,v) \atop j \right) \left(n-d(u,v) \atop k-j \right)\\
\nonumber& = & \sum_{\substack{u, v \in \mathcal C\\u \not = v}} \frac{2(q-1)}{q}P_k^-(n; d(u,v))\\
\nonumber& = & \frac{2(q-1)}{q} \sum_{i=1}^n \sum_{\substack{u, v \in \mathcal C\\d(u,v)=i}} P_k^-(n; i)\\
& = & \frac{2(q-1)M}{q} \sum_{i=1}^n P_k^-(n;i) B_i.
\end{eqnarray}
Now for each $\alpha = (\alpha_1, \alpha_2, \ldots, \alpha_k) \in (\mathbb{F}_q^*)^k$ and $k$ columns $u'_{i_1}, u'_{i_2}, \ldots, u'_{i_k}$ $(i_1 < i_2 < \cdots < i_k)$, there are
\begin{eqnarray}
\sum_{c<d} 2 x_c(\alpha_1 u'_{i_1} + \cdots + \alpha_k u'_{i_k}) \cdot x_d(\alpha_1 u'_{i_1} + \cdots + \alpha_k u'_{i_k})
\end{eqnarray}
choices for $m\not = l$ such that
$\alpha_1 c_{mi_1} + \cdots + \alpha_k c_{mi_k} \not = \alpha_1 c_{li_1} + \cdots + \alpha_k c_{li_k}$.
It follows that
\begin{eqnarray}\label{hh}
\nonumber S_1(k) & = & \sum_{\substack{\alpha \in (\mathbb{F}_q^*)^k\\ i_1 < \cdots < i_k}} \sum_{\substack{c<d }}2 x_c(\alpha_1 u'_{i_1} + \cdots + \alpha_k u'_{i_k}) \cdot x_d(\alpha_1 u'_{i_1} + \cdots + \alpha_k u'_{i_k})\\
& = & 2S(k).
\end{eqnarray}
(\ref{mm}) and (\ref{hh}) give
\begin{eqnarray}
\sum_{i=1}^n P_k^-(n;i)B_i = \frac{q}{(q-1)M}S(k).
\end{eqnarray}
Finally,
\begin{eqnarray}
\nonumber \sum_{i=1}^n P_k^+(n;i)B_i & = & \sum_{i=1}^n \left[ (q-1)^k \left(n \atop k \right) - P_k^-(n;i) \right] B_i\\
\nonumber & = & \sum_{i=1}^n (q-1)^k \left( n \atop k \right) B_i - \sum_{i=1}^n P_k^-(n;i) B_i\\
& = & (M-1)(q-1)^k \left(n \atop k \right) - \frac{q}{(q-1)M}S(k).
\end{eqnarray}
\end{proof}


\begin{theorem}(Improved Delsarte inequalities).\label{delthm}
Suppose that $\mathcal C$ is a $q$-ary code of length $n$ with size $M$ and distance distribution $\{B_i\}_{i=0}^n$. If $r$ is the remainder when dividing $M$ by $q$, then for each $k = 1, 2, \ldots, n$,
\begin{eqnarray}
\sum_{i=0}^n P_k(n;i) B_i \geq \frac{1}{M} r(q-r) (q-1)^{k-1} \left(n \atop k\right).
\end{eqnarray}
\end{theorem}

\begin{proof}
By Lemmas \ref{qt2} and \ref{qt1},
\begin{eqnarray}
\nonumber - \sum_{i=1}^n P_k(n;i)B_i & = & \sum_{i=1}^n P_k^-(n;i)B_i - \sum_{i=1}^n P_k^+(n;i)B_i\\
\nonumber & = & -(M-1)(q-1)^k \left(n \atop k \right) + \frac{2q}{(q-1)M}S(k)\\
\nonumber & \leq & -(M-1)(q-1)^k \left(n \atop k \right) \\
\nonumber && + \frac{2q}{(q-1)M}(q-1)^k \left(n \atop k \right) \left[ \frac{q-1}{2q} M^2 + \frac{r(r-q)}{2q} \right]\\
&=& (q-1)^k \left(n \atop k \right) - \frac{1}{M} r(q-r)(q-1)^{k-1}\left(n \atop k \right).
\end{eqnarray}
Since $P_k(n;0) = (q-1)^k \left(n \atop k \right)$ and $B_0=1$, the above inequality gives the desired result.
\end{proof}

\section{Inequalities for $q$-Ary Constant-Weight Codes}\label{ss3}

Let $\mathcal C$ be a $q$-ary constant-weight code of length $n$ and constant-weight $w$. Let $M = |\mathcal C|$. As before, consider $\mathcal C$ as an $M \times n$ matrix (where each $c \in \mathcal C$ is a row). The $m$th row of $\mathcal C$ is denoted by $(c_{m1}, c_{m2}, \ldots, c_{mn})$, $m = 1, 2, \ldots, M$. Let $u'_1, u'_2, \ldots, u'_n$ be the $n$ columns of $\mathcal C$. Write $\mathbb{F}_q = \{0=\omega_1, \omega_2, \ldots, \omega_q\}$.

\begin{lemma}\label{g1}
For each $k = 1, 2, \ldots, n$,
\begin{eqnarray}
\sum_{\substack{\alpha \in (\mathbb{F}_q^*)^k \\ i_1 < \cdots < i_k}} \sum_{\substack{2 \leq c \leq q}} x_c(\alpha_1 u'_{i_1} + \cdots + \alpha_k u'_{i_k}) = \frac{2(q-1)}{q} M P_k^-(n;w).
\end{eqnarray}
\end{lemma}

\begin{proof}
Let $S_0(k)$ be the number of pairs $(B,\alpha)$ satisfying the following conditions.
\begin{itemize}
\item [(a)] $B = (c_{mi_1}~c_{mi_2} \cdots c_{mi_k})$, where $m = 1, 2, \ldots, M$ and $i_1 < i_2 < \cdots < i_k$.
\item [(b)] $\alpha = (\alpha_1, \alpha_2, \ldots, \alpha_k) \in (\mathbb{F}_q^*)^k$ such that
\begin{eqnarray}
\alpha_1 c_{mi_1} + \alpha_2 c_{mi_2} + \cdots + \alpha_k c_{mi_k} \not = 0.
\end{eqnarray}
\end{itemize}
Let $u = (c_{m1}~c_{m2} \cdots c_{mn})$ be a row of $\mathcal C$. In the proof of Lemma \ref{qt2}, if $v$ is the zero vector, then the number of $\alpha \in (\mathbb{F}_q^*)^k$ and $i_1 < i_2 < \cdots < i_k$ such that $\alpha_1 c_{mi_1} + \alpha_2 c_{mi_2} + \cdots + \alpha_k c_{mi_k} \not = 0$ is
\begin{eqnarray}
\sum_{j=0}^w \frac{q-1}{q}[(q-1)^j - (-1)^j] (q-1)^{k-j}\left(w \atop j \right) \left( n-w \atop j \right) = \frac{2(q-1)}{q} P_k^-(n;w).
\end{eqnarray}
Therefore,
\begin{eqnarray}\label{toan1}
S_0(k) = \frac{2(q-1)M}{q} P_k^-(n;w).
\end{eqnarray}
Now for each $\alpha = (\alpha_1, \alpha_2, \ldots, \alpha_k) \in (\mathbb{F}_q^*)^k$ and $k$ columns $u'_{i_1}, u'_{i_2}, \ldots, u'_{i_k}$ $(i_1 < i_2 < \cdots < i_k)$, there are exactly
\begin{eqnarray}
\sum_{2 \leq c \leq q} x_c(\alpha_1 u'_{i_1} + \cdots + \alpha_k u'_{i_k})
\end{eqnarray}
choices for $m$ such that $\alpha_1 c_{mi_1} + \alpha_2 c_{mi_2} + \cdots + \alpha_k c_{mi_k} \not = 0.$
Hence,
\begin{eqnarray}\label{toan2}
S_0(k) = \sum_{\substack{\alpha \in (\mathbb{F}_q^*)^k \\ i_1 < \cdots < i_k}} \sum_{\substack{2 \leq c \leq q}} x_c(\alpha_1 u'_{i_1} + \cdots + \alpha_k u'_{i_k}).
\end{eqnarray}
(\ref{toan1}) and (\ref{toan2}) gives the desired result.
\end{proof}

\begin{lemma}\label{lmmatran}
Suppose that $A = (n_{ci})$ is a $q \times N$ matrix with nonnegative entries such that
$\sum_{i=1}^N \sum_{c=2}^q n_{ci} = M'$ and $\sum_{c=1}^q n_{ci} = M$ for all $i = 1, 2, \ldots, N$, where $M'$ and $M$ are constants. Then the sum
\begin{eqnarray}
\sum_{i=1}^N \sum_{c<d} n_{ci}n_{di}
\end{eqnarray}
is maximum if and only if $|n_{1i} - n_{1j}| \leq 1$ for all $i \not = j$ and $|n_{ci}- n_{di}| \leq 1$ for all $i = 1, 2, \ldots, N$ and all $2 \leq c<d \leq q$.
\end{lemma}

\begin{proof}
Suppose that there exist $i_0 \not = j_0$ such that $n_{1i_0} - n_{1j_0} > 1$. Since $\sum_{c=1}^q n_{ci_0} = M = \sum_{c=1}^q n_{cj_0}$, there exists $c_0$ $(2 \leq c_0 \leq q)$ such that $n_{c_0 j_0}-n_{c_0 i_0} \geq 1$. Let $A'= (n'_{ci})$ be the $q \times N$ matrix defined by $n'_{1i_0}=n_{1i_0}-1$, $n'_{c_0i_0}=n_{c_0i_0} + 1$, $n'_{1j_0}=n_{1j_0}+1$, $n'_{c_0j_0}=n_{c_0j_0} - 1$, and the other entries of $A'$ are the same as those of $A$. We have $\sum_{i=1}^N \sum_{c=2}^q n'_{ci} = M'$ and $\sum_{c=1}^q n'_{ci} = M$ for all $i = 1, 2, \ldots, N$. Furthermore,
\begin{eqnarray}
\nonumber \sum_{i=1}^N \sum_{c<d} n'_{ci}n'_{di} - \sum_{i=1}^N \sum_{c<d} n_{ci}n_{di} & = & \sum_{i \in \{i_0, j_0\}} \sum_{c<d} n'_{ci}n'_{di} - \sum_{i\in \{i_0, j_0\}}^N \sum_{c<d} n_{ci}n_{di}\\
\nonumber & = & \sum_{i \in \{i_0, j_0\}} n'_{1i}n'_{c_0i} - \sum_{i\in \{i_0, j_0\}}^N n_{1i}n_{c_0i}\\
\nonumber & = & n'_{1i_0}n'_{c_0i_0} + n'_{1j_0}n'_{c_0j_0}-n_{1i_0}n_{c_0i_0}-n_{1j_0}n_{c_0j_0}\\
\nonumber &=&(n_{1i_0}-1) (n_{c_0i_0}+1) + (n_{1j_0}+1)(n_{c_0j_0}-1) \\
\nonumber && -n_{1i_0}n_{c_0i_0}-n_{1j_0}n_{c_0j_0}\\
\nonumber & = &n_{1i_0} - n_{c_0i_0} - 1 - n_{1j_0} + n_{c_0 j_0} -1\\
\nonumber & = & (n_{1i_0} - n_{1j_0}) + (n_{c_0j_0} - n_{c_0i_0}) - 2 \\
& > & 1 + 1 - 2 = 0.
\end{eqnarray}
Similarly, suppose that there exist $i_0$ $(1 \leq i_0 \leq N)$ and $c_0 \not= d_0$ such that $c_0 \geq 2$, $d_0 \geq 2$, and $n_{c_0i_0} - n_{d_0i_0}>1$. Then let $A''$ be the matrix defined by $n''_{c_0i_0} = n_{c_0i_0} - 1$, $n''_{d_0i_0} = n_{d_0i_0} + 1$, and the other entries of $A''$ are the same as those of $A$. We have $\sum_{i=1}^N \sum_{c=2}^q n''_{ci} = M'$ and $\sum_{c=1}^q n''_{ci} = M$ for all $i = 1, 2, \ldots, N$. Furthermore,
\begin{eqnarray}
\nonumber \sum_{i=1}^N \sum_{c<d} n''_{ci}n''_{di} - \sum_{i=1}^N \sum_{c<d} n_{ci}n_{di} & = & \sum_{c<d} n''_{ci_0}n''_{di_0} - \sum_{c<d} n_{ci_0}n_{di_0}\\
\nonumber & = & n''_{c_0i_0}n''_{d_0i_0} - n_{c_0i_0}n_{d_0i_0}\\
\nonumber & = & (n_{c_0i_0}-1) (n_{d_0i_0}+1) - n_{c_0i_0}n_{d_0i_0}\\
& = &n_{c_0i_0} - n_{d_0i_0} - 1 > 0.
\end{eqnarray}
\end{proof}

Let $\mathcal C$ be a $q$-ary constant-weight code of length $n$ and constant-weight $w$. Let $M = |\mathcal C|$.
For each $k = 1, 2, \ldots, n$, denote
\begin{eqnarray}
T_1(k) = \left[(q-1)^k \left(n \atop k \right) - r_k \right](M-q_k)q_k + r_k(M-q_k-1)(q_k+1),
\end{eqnarray}
\begin{eqnarray}
\nonumber T_2(k) & = & \left[(q-1)^k \left( n \atop k \right) - r_k \right] \left[\left(q-1-t_k \atop 2 \right) s_k^2  \right.\\
&& \quad\quad\quad\quad\quad\quad \left.+  (q-1-t_k)t_ks_k (s_k+1) + \left(t_k \atop 2 \right) (s_k+1)^2 \right],
\end{eqnarray}
and
\begin{eqnarray}
T_3(k) =  r_k \left[ \left(q-1-t_k' \atop 2 \right) s_k'^2 + (q-1-t_k')t_k's_k'(s_k'+1) + \left(t_k' \atop 2 \right) (s_k'+1)^2 \right],
\end{eqnarray}
where
\begin{itemize}
\item [$\circ$] $q_k$ and $r_k$ are the quotient and the remainder, respectively, when dividing $\frac{2(q-1)M}{q} P_k^-(n;w)$ by $(q-1)^k \left( n \atop k \right)$,
\item [$\circ$] $s_k$ and $t_k$ are the quotient and the remainder, respectively, when dividing $q_k$ by $(q-1)$,
\item [$\circ$] $s_k'$ and $t_k'$ are the quotient and the remainder, respectively, when dividing $q_k+1$ by $(q-1)$.
\end{itemize}

\begin{lemma}\label{3t}
Suppose that $\mathcal C$ is a constant-weight code of length $n$ and constant-weight $w$. Let $S(k)$ be defined by (\ref{tm10}). Then for each $k = 1, 2, \ldots, n$,
\begin{eqnarray}
S(k) \leq T_1(k) + T_2(k) + T_3(k).
\end{eqnarray}
\end{lemma}

\begin{proof}
Let $N = (q-1)^k \left(n \atop k \right)$ and let $A$ be the $q \times N$ matrix defined as follows.
\begin{itemize}
\item [$\circ$] The rows of $A$ are indexed by $c$ $(c=1,2, \ldots, q)$.
\item [$\circ$] The columns of $A$ are indexed by pairs $(\alpha, i_1 < \cdots < i_k)$, where $\alpha \in (\mathbb{F}_q^*)^k$ and $1 \leq i_1 < \cdots < i_k \leq q$.
\item [$\circ$] The $(c,(\alpha, i_1 < \cdots < i_k))$ entry of $A$ is $x_c(\alpha_1 u'_{i_1} + \cdots + \alpha_k u'_{i_k})$.
\end{itemize}
Recall that
\begin{eqnarray}
S(k)=\sum_{\substack{\alpha \in (\mathbb{F}_q^*)^k \\ i_1< \cdots < i_k}} \sum_{\substack{c<d}} x_c(\alpha_1 u'_{i_1} + \cdots + \alpha_k u'_{i_k}) \cdot x_d(\alpha_1 u'_{i_1} + \cdots + \alpha_k u'_{i_k}).
\end{eqnarray}
By Lemma \ref{g1},
\begin{eqnarray}
\sum_{\substack{\alpha \in (\mathbb{F}_q^*)^k \\ i_1 < \cdots < i_k}} \sum_{\substack{2 \leq c \leq q}} x_c(\alpha_1 u'_{i_1} + \cdots + \alpha_k u'_{i_k}) = \frac{2(q-1)}{q} M P_k^-(n;w).
\end{eqnarray}
Also, by definition,
\begin{eqnarray}
\sum_{1 \leq c \leq q}x_c(\alpha_1 u'_{i_1} + \cdots + \alpha_k u'_{i_k}) = M
\end{eqnarray}
Hence, we can apply Lemma \ref{lmmatran} to the matrix $A$. Lemma \ref{lmmatran} implies that $S(k)$ is maximum when
\begin{eqnarray}
\sum_{2 \leq c \leq q}x_c(\alpha_1 u'_{i_1} + \cdots + \alpha_k u'_{i_k}) = q_k \quad \mbox{ or } \quad x_1(\alpha_1 u'_{i_1} + \cdots + \alpha_k u'_{i_k}) = M-q_k
\end{eqnarray}
for $(q-1)^k \left(n \atop k \right) - r_k$ pairs $(\alpha, i_1 < \cdots < i_k)$ and
\begin{eqnarray}
\sum_{2 \leq c \leq q}x_c(\alpha_1 u'_{i_1} + \cdots + \alpha_k u'_{i_k}) = q_k+1 \\
\mbox{ or } \quad x_1(\alpha_1 u'_{i_1} + \cdots + \alpha_k u'_{i_k}) = M-q_k-1
\end{eqnarray}
for the other $r_k$ pairs $(\alpha, i_1 < \cdots < i_k)$. Furthermore, for each pair $(\alpha, i_1 < \cdots < i_k)$, the following must hold.
\begin{itemize}
\item [(i)] If $\sum_{2 \leq c \leq q}x_c(\alpha_1 u'_{i_1} + \cdots + \alpha_k u'_{i_k}) = q_k$, then
\begin{eqnarray}
x_e(\alpha_1 u'_{i_1} + \cdots + \alpha_k u'_{i_k}) = s_k
\end{eqnarray}
for $q-1-t_k$ values of $e \geq 2$ and
\begin{eqnarray}
x_e(\alpha_1 u'_{i_1} + \cdots + \alpha_k u'_{i_k}) = s_k+1
\end{eqnarray}
for the other $t_k$ values of $e \geq 2$.
\item [(ii)] If $\sum_{2 \leq c \leq q}x_c(\alpha_1 u'_{i_1} + \cdots + \alpha_k u'_{i_k}) = q_k+1$, then
\begin{eqnarray}
x_e(\alpha_1 u'_{i_1} + \cdots + \alpha_k u'_{i_k}) = s_k'
\end{eqnarray}
for $q-1-t_k'$ values of $e \geq 2$ and
\begin{eqnarray}
x_e(\alpha_1 u'_{i_1} + \cdots + \alpha_k u'_{i_k}) = s_k'+1
\end{eqnarray}
for the other $t_k'$ values of $e \geq 2$.
\end{itemize}
Therefore,
\begin{eqnarray}
\nonumber S(k) & = & \sum_{\substack{\alpha \in (\mathbb{F}_q^*)^k \\i_1< \cdots < i_k}} \sum_{\substack{ 2\leq d \leq q}} x_1(\alpha_1 u'_{i_1} + \cdots + \alpha_k u'_{i_k}) \cdot x_d(\alpha_1 u'_{i_1} + \cdots + \alpha_k u'_{i_k})\\
\nonumber & & + \sum_{\substack{\alpha \in (\mathbb{F}_q^*)^k\\i_1< \cdots < i_k}} \sum_{\substack{2 \leq c<d \leq q}} x_c(\alpha_1 u'_{i_1} + \cdots + \alpha_k u'_{i_k}) \cdot x_d(\alpha_1 u'_{i_1} + \cdots + \alpha_k u'_{i_k})\\
\nonumber & \leq & \left[(q-1)^k \left(n \atop k \right) - r_k \right](M-q_k)q_k + r_k(M-q_k-1)(q_k+1)\\
\nonumber && +\left[(q-1)^k \left( n \atop k \right) - r_k \right] \left[\left(q-1-t_k \atop 2 \right) s_k^2  \right.\\
\nonumber && \quad\quad\quad\quad\quad\quad\quad\quad\quad\quad \left.+  (q-1-t_k)t_ks_k (s_k+1) + \left(t_k \atop 2 \right) (s_k+1)^2 \right]\\
\nonumber && +r_k \left[ \left(q-1-t_k' \atop 2 \right) s_k'^2 + (q-1-t_k')t_k's_k'(s_k'+1) + \left(t_k' \atop 2 \right) (s_k'+1)^2 \right]\\
&=& T_1(k) + T_2(k) + T_3(k).
\end{eqnarray}
\end{proof}

\begin{theorem}\label{maintr}
Suppose that $\{B_i\}_{i=0}^n$ is the distance distribution of a $q$-ary constant-weight code $\mathcal C$ of length $n$ and constant-weight $w$. Then for each $k = 1, 2, \ldots, n$,
\begin{eqnarray}
\sum_{i=1}^n P_k(n;i)B_i \geq (M-1)(q-1)^k\left(n \atop k \right) - \frac{2q}{(q-1)M}T(k),
\end{eqnarray}
where $T(k) = T_1(k) + T_2(k) + T_3(k)$.
\end{theorem}

\begin{proof}
Considering $\mathcal C$ as a $q$-ary code and applying Lemma \ref{qt2}, we get
\begin{eqnarray}
\nonumber - \sum_{i=1}^n P_k(n;i)B_i & = & \sum_{i=1}^n P_k^-(n;i)B_i - \sum_{i=1}^n P_k^+(n;i)B_i\\
& = & -(M-1)(q-1)^k \left(n \atop k \right) + \frac{2q}{(q-1)M}S(k).
\end{eqnarray}
By Lemma \ref{3t},
\begin{eqnarray}
S(k) \leq T_1(k) + T_2(k) + T_3(k) = T(k).
\end{eqnarray}
Hence,
\begin{eqnarray}
\nonumber - \sum_{i=1}^n P_k(n;i)B_i  \leq -(M-1)(q-1)^k \left(n \atop k \right) + \frac{2q}{(q-1)M}T(k).
\end{eqnarray}
\end{proof}

When $k=1$, Theorem \ref{maintr} implies the following corollary, which appeared in \cite{o} (see \cite[Theorem 12]{o}).

\begin{corollary}
If there exists a $q$-ary constant-weight code of length $n$, constant-weight $w$, and minimum distance $\geq d$, then
\begin{eqnarray}
M(M-1)d \leq  2t \sum_{i=0}^{q-2} \sum_{j=i+1}^{q-1} M_i M_j + 2(n-t) \sum_{i=0}^{q-2} \sum_{j=i+1}^{q-1} M'_i M'_j,
\end{eqnarray}
where
\begin{itemize}
\item [$\circ$] $k$ and $t$ are the quotient and the remainder, respectively, when dividing $Mw$ by $n$,
\item [$\circ$] $M_0 = M - k - 1$, $M'_0 = M-k$, $M_i = \lfloor (k+i)/(q-1) \rfloor$, and $M'_i = \lfloor (k+i-1)/(q-1) \rfloor$.
\end{itemize}
\end{corollary}

\begin{proof}
Applying Theorem \ref{maintr} (for $k = 1$), we have
\begin{eqnarray}\label{tamthuong1}
\sum_{i=1}^n [(q-1)n - qi]B_i \geq (M-1)(q-1)n- \frac{2q}{(q-1)M}T(1).
\end{eqnarray}
Since the code has minimum distance $\geq d$, $B_i = 0$ if $1 < i < d$. Multiplying two sides of (\ref{tamthuong1}) by $M$, we get
\begin{eqnarray}
\nonumber M(M-1)(q-1)n- \frac{2q}{(q-1)}T(1) & \leq & M \sum_{i=1}^n [(q-1)n - qi]B_i\\
\nonumber & = & M \sum_{i=d}^n [(q-1)n - qi]B_i \\
\nonumber & \leq & M \sum_{i=d}^n [(q-1)n - qd]B_i \\
& = & M(M-1)[(q-1)n - qd].
\end{eqnarray}
This implies
\begin{eqnarray}\label{aa1}
M(M-1)d \leq \frac{2}{q-1}T(1).
\end{eqnarray}
By hypothesis, $Mw = kn + t$ $(0 \leq t <n)$. Hence, if $q_1$ and $r_1$ are the quotient and the remainder, respectively, when dividing $\frac{2(q-1)M}{q}P_1^-(n;w) = (q-1)Mw$ by $(q-1)n$, then $q_1 = (q-1)k$ and $r_1 = (q-1)t$. We also have $s_1 = k$, $t_1 = 0$, $s'_1 = k$ and $t'_1 =1$. Now it is not difficult to see that
\begin{eqnarray}\label{aa2}
\nonumber T(1) & = & T_1(1) + T_2(1) + T_3(1) \\
\nonumber & = &\{[(q-1)n - r_1] (M-q_1)q_1 + T_2(1)\} + \{r_1(M-q_1-1)(q_1+1) + T_3(1)\}\\
& = & (q-1)(n-t)\sum_{i=1}^{q-2} \sum_{j=i+1}^{q-1} M'_i M'_j + (q-1)t\sum_{i=1}^{q-2} \sum_{j=i+1}^{q-1} M_i M_j.
\end{eqnarray}
From (\ref{aa1}) and (\ref{aa2}), we get the desired result
\begin{eqnarray}
M(M-1)d \leq  2t \sum_{i=0}^{q-2} \sum_{j=i+1}^{q-1} M_i M_j + 2(n-t) \sum_{i=0}^{q-2} \sum_{j=i+1}^{q-1} M'_i M'_j.
\end{eqnarray}
\end{proof}

\end{document}